\documentclass[reqno]{amsart}
\usepackage[colorlinks,linkcolor=black,anchorcolor=black,citecolor=black,hyperindex=black,CJKbookmarks=black]{hyperref}
\usepackage{float}
\usepackage{graphicx}
\usepackage{graphics}
\usepackage{subfigure}
\usepackage{mathrsfs}

\newtheorem{theorem}{Theorem}[section]
\newtheorem{lemma}[theorem]{Lemma}
\newtheorem{proposition}[theorem]{Proposition}
\newtheorem{definition}[theorem]{Definition}

\theoremstyle{remark}
\newtheorem{remark}[theorem]{Remark}

\numberwithin{equation}{section}

\begin{document}

\title[A remark on the extreme value theory for continued fractions]{A remark on the extreme value theory for continued fractions}

\author {Lulu Fang}
\address{School of Mathematics, South China University of Technology, Guangzhou 510640, P.R. China}
\email{f.lulu@mail.scut.edu.cn}

\author {Kunkun Song}
\address{School of Mathematics and Statistics, Wuhan University, Wuhan, 430072, P.R. China}
\email{songkunkun@whu.edu.cn}


\subjclass[2010]{Primary 11K50, 28A80; Secondary 60G70}
\keywords{Continued fractions, Extreme value theory, Doubly exponential rate, Hausdorff dimension}

\begin{abstract}
Let $x$ be a irrational number in the unit interval and denote by its continued fraction expansion $[a_1(x), a_2(x), \cdots, a_n(x), \cdots]$. For any $n \geq 1$, write $T_n(x) = \max_{1 \leq k \leq n}\{a_k(x)\}$. We are interested in the Hausdorff dimension of the fractal set
\[
E_\phi = \left\{x \in (0,1): \lim_{n \to \infty} \frac{T_n(x)}{\phi(n)} =1\right\},
\]
where $\phi$ is a positive function defined on $\mathbb{N}$ with $\phi(n) \to \infty$ as $n \to \infty$. Some partial results have been obtained by Wu and Xu, Liao and Rams, and Ma. In the present paper, we further study this topic when $\phi(n)$ tends to infinity with a doubly exponential rate as $n$ goes to infinity.

\end{abstract}

\maketitle

\section{Introduction}
Every irrational number $x$ in the unit interval has a unique \emph{continued fraction expansion} of the form
\begin{equation}\label{CF}
x =  \dfrac{1}{a_1(x) +\dfrac{1}{a_2(x) + \ddots +\dfrac{1}{a_n(x)+ \ddots}}},
\end{equation}
where $a_n(x)$ are positive integers and are called the \emph{partial quotients} of the continued fraction expansion of $x$~$(n \in \mathbb{N})$. Sometimes we write the representation (\ref{CF}) as $[a_1(x), a_2(x), \cdots, a_n(x), \cdots]$.
For more details about continued fractions, we refer the reader to a monograph of Khintchine \cite{lesKhi64}.

Let $x \in (0,1)$ be an irrational number. For any $n \geq 1$, we define
\[
T_n(x) = \max\{a_k(x):1 \leq k \leq n\},
\]
i.e., the largest one in the block of the first $n$ partial quotients of the continued fraction expansion of $x$. Extreme value theory in probability theory is concerned with the limit distribution laws for the maximum of a sequence of random variables (see \cite{lesLLR83}). Galambos \cite{lesGal72} first considered the extreme value theory for continued fractions and obtained that
\[
\lim_{n \to \infty} \mu\left\{x \in (0,1): \frac{T_n(x)}{n}< \frac{y}{\log 2} \right\} = e^{-1/y}
\]
for any $y >0$, where $\mu$ is equivalent to the Lebesgue measure, namely \emph{Gauss measure} given by
\[
\mu(A) = \frac{1}{\log 2}\int_A \frac{1}{1+x}dx
\]
for any Borel set $A \subseteq (0, 1)$. Later, he also gave an iterated logarithm type theorem for $T_n(x)$ in \cite{lesGal73}, that is, for $\mu$-almost all $x \in (0,1)$,
\[
\limsup_{n \to \infty} \frac{\log T_n(x) -\log n}{\log \log n} =1\ \ \ \ \text{and}\ \ \ \
\liminf_{n \to \infty} \frac{\log T_n(x) -\log n}{\log \log n} =0.
\]
As a consequence, we know that
\[
\lim_{n \to \infty} \frac{\log T_n(x)}{\log n} =1
\]
holds for $\mu$-almost all $x \in (0,1)$.
Furthermore, Philipp \cite{lesPhi75} solved a conjecture of Erd\H{o}s for $T_n(x)$ and obtained its order of magnitude. In fact,
\[
\liminf_{n \to \infty} \frac{T_n(x) \log \log n}{n} = \frac{1}{\log 2}.
\]
holds for $\mu$-almost all $x \in (0,1)$. Later, Okano \cite{lesOka02} constructed some explicit real numbers that satisfied this liminf.
However, a natural question arises: what is the exact growth rate of $T_n$ or
whether there exists a normalizing sequence $\{b_n\}_{n\geq1}$ such that $T_n(x)/b_n$ converges to a positive and finite constant for $\mu$-almost all $x \in (0,1)$. Unfortunately, there is a negative answer for this question. That is to say, there is no such a normalizing non-decreasing sequence $\{b_n\}_{n\geq1}$ so that $T_n(x)/b_n$ converges to a positive and finite constant for $\mu$-almost all $x \in (0,1)$. More precisely,
\[
\limsup_{n \to \infty} \frac{T_n(x)}{b_n} = 0\ \ \ \ \ \ \ \ \text{or}\ \ \ \ \ \ \ \ \limsup_{n \to \infty} \frac{T_n(x)}{b_n} = +\infty
\]
holds for $\mu$-almost all $x \in (0,1)$ according to $\sum_{n \geq 1} 1/b_n$ converges or diverges.
In 1935, Khintchine \cite{lesKhi35} proved that $S_n(x)/(n\log n)$ converges in measure to the constant $1/(\log 2)$ and Philipp \cite{lesPhi88} remarked that this result cannot hold for $\mu$-almost all $x \in (0,1)$, where $S_n(x) = a_1(x)+ \cdots + a_n(x)$. That is to say, the strong law of large numbers for $S_n$ fails.
However,
Diamond and Vaaler \cite{lesDV86} showed that the maximum $T_n(x)$ should be responsible for the failure of the strong law of large numbers, i.e.,
\[
\lim_{n \to \infty} \frac{S_n(x) - T_n(x)}{n\log n}  = \frac{1}{\log 2}
\]
for $\mu$-almost all $x \in (0,1)$. These results indicate that the maximum $T_n(x)$ play an important role in metric theory of continued fractions.
For more metric results on extreme value theory for continued fractions, we refer the reader to Barbolosi \cite{lesBar99}, Bazarova et al.~\cite{lesBBH16},  Iosifescu and Kraaikamp \cite{lesIK02} and Kesseb\"{o}hmer and Slassi \cite{lesKS08, lesKS08MN}.




The following will study the maximum $T_n(x)$ from the viewpoint of the fractal dimension. More precisely, we
are interested in the Hausdorff dimension of the fractal set
\[
E_\phi = \left\{x \in (0,1): \lim_{n \to \infty} \frac{T_n(x)}{\phi(n)} =1\right\},
\]
where $\phi$ is a positive function defined on $\mathbb{N}$ with $\phi(n) \to \infty$ as $n \to \infty$. Wu and Xu \cite{lesWX09} have obtained some partial results on this topic and they pointed out that $E_\phi$ has full Hausdorff dimension if $\phi(n)$ tends to infinity with a polynomial rate as $n$ goes to infinity.

\begin{theorem}[\cite{lesWX09}]
Assume that $\phi$ is a positive function defined on $\mathbb{N}$ satisfying $\phi(n) \to \infty$ as $n \to \infty$ and
\[
\lim_{n \to \infty} \frac{\log \phi(n)}{\log n} < \infty.
\]
Then
\begin{equation*}
\dim_{\rm H}\left\{x \in (0,1): \lim_{n \to \infty} \frac{T_n(x)}{\phi(n)} =1\right\} =1.
\end{equation*}
\end{theorem}

Liao and Rams \cite{lesLR16} considered the Hausdorff dimension of $E_\phi$ when $\phi(n)$ tends to infinity with a single exponential rate. Also, they showed that there is a jump of the Hausdorff dimensions from 1 to 1/2 on the class $\phi(n) = e^{n^\alpha}$ at $\alpha=1/2$.

\begin{theorem}[\cite{lesLR16}]
For any $\beta >0$,
\begin{equation*}
\dim_{\rm H}\left\{x \in (0,1): \lim_{n \to \infty} \frac{T_n(x)}{e^{n^\alpha}} =\beta\right\} =
\begin{cases}
1,  &\text{$\alpha \in (0,1/2)$};\\
1/2,  & \text{$\alpha \in (1/2,+\infty)$}.
\end{cases}
\end{equation*}
\end{theorem}

However, they don't know what will happen at the critical point $\alpha=1/2$. Recently, Ma \cite{lesMa} solved this left unknown problem and proved that
its Hausdorff dimension is 1/2 when $\alpha=1/2$.

In the present paper, we will further investigate the Hausdorff dimension of $E_\phi$ when $\phi(n)$ tends to infinity with a doubly exponential rate. Moreover, we will see that the Hausdorff dimension of $E_\phi$ will decay to zero if the speed of $\phi(n)$ is growing faster and faster, which can be treated as a supplement to Wu and Xu \cite{lesWX09}, Liao and Rams \cite{lesLR16}, and Ma \cite{lesMa} in this topic.


\begin{theorem}\label{doubly exponential}
Let $b,c >1$ be real numbers. Then for any $\beta > 0$,
\begin{equation*}
\dim_{\rm H} \left\{x \in (0,1): \lim_{n \to \infty} \frac{T_n(x)}{c^{b^{n^\alpha}}}= \beta\right\}=
\begin{cases}
1/2,  &\text{$\alpha \in (0,1)$};\\
1/(b+1),  &\text{$\alpha=1$};\\
0,  & \text{$\alpha \in (1,+\infty)$}.
\end{cases}
\end{equation*}
\end{theorem}

For more results about the Hausdorff dimensions of fractal sets related to $T_n$, see Zhang \cite{lesZhang16}, and Zhang and L\"{u} \cite{lesZL16}.
The following figure is an illustration of the Haudorff dimension of $E_\phi$ for different $\phi$.

\begin{figure}[H]
\centering
\includegraphics[height=1.8in, width=4.5in]{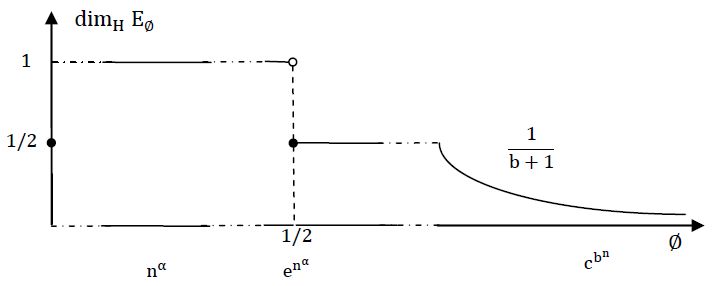}
\caption{{$\dim_{\rm H} E_\phi$} for different $\phi$}
\end{figure}

\section{Preliminaries}
This section is devoted to recalling some definitions and basic properties of continued fractions.

Let $x \in (0,1)$ be a irrational number and its continued fraction expansion $x=[a_1(x),a_2(x),\cdots,a_n(x),\cdots]$. For any $n \geq 1$, we denote by
\[
\frac{p_n(x)}{q_n(x)}:= [a_1(x), a_2(x), \cdots, a_n(x)]
\]
the $n$-th \emph{convergent} of the continued fraction expansion of $x$, where $p_n(x)$ and $q_n(x)$ are relatively prime.
With the conventions $p_{-1}=1$, $q_{-1}=0$, $p_0=0$, $q_0=1$, the quantities $p_n(x)$ and $q_n(x)$ satisfy the following recursive formula:
\begin{equation}\label{recursive}
p_n(x) = a_n(x) p_{n-1}(x) + p_{n-2}(x)\ \ \ \text{and}\ \ \ q_n(x) = a_n(x) q_{n-1}(x) + q_{n-2}(x),
\end{equation}
which implies that
\begin{equation}\label{DX}
a_1(x)a_2(x)\cdots a_n(x) \leq q_n(x) \leq (a_1(x)+1)(a_2(x)+1) \cdots (a_n(x)+1).
\end{equation}

\begin{definition}
For any $n \geq 1$ and $a_1, a_2, \cdots, a_n\in \mathbb{N}$, we call
\begin{equation*}
I(a_1, \cdots, a_n):= \left\{x \in (0,1): a_1(x)=a_1, \cdots, a_n(x)=a_n\right\}
\end{equation*}
the $n$-th order \emph{cylinder} of continued fraction expansion.
\end{definition}
In other words, $I(a_1, \cdots, a_n)$ is the set of points beginning with $(a_1 ,a_2, \cdots, a_n)$ in their continued fraction expansions.

\begin{proposition}\label{CY}
Let $n \geq 1$ and $a_1, a_2, \cdots, a_n \in \mathbb{N}$. Then $I(a_1, \cdots, a_n)$ is an interval with two endpoints
\[
\frac{p_n}{q_n}\ \ \ \ \ \text{and}\ \ \ \ \  \frac{p_n+p_{n-1}}{q_n+q_{n-1}}.
\]
More precisely, $p_n/q_n$ is the left endpoint if $n$ is even; otherwise it is the right endpoint. Moreover, the length of $I(a_1, \cdots, a_n)$ satisfies
\begin{equation}\label{cylinder inequality}
|I(a_1, \cdots, a_n)| = \frac{1}{q_n(q_n+q_{n-1})},
\end{equation}
where $p_n$ and $q_n$ satisfy the recursive formula (\ref{recursive}).
\end{proposition}

\section{Proof of Theorem \ref{doubly exponential}}

In this section, we will give the proof of Theorem \ref{doubly exponential} which is inspired by Xu \cite{lesXu08}.
More precisely, we will determine the Hausdorff dimension of the fractal set
\[
E(b,c,\alpha,\beta) = \left\{x \in (0,1): \lim_{n \to \infty} \frac{T_n(x)}{c^{b^{n^\alpha}}}= \beta\right\}
\]
for $b,c >1$ and $\alpha,\beta>0$. The proof is divided into two parts: upper bound and lower bound for $\dim_{\rm H} E(b,c,\alpha,\beta)$.

\subsection{Upper bound}
Let $b,c>1$ be real numbers. For any $\alpha>0$, we define
\[
E^\ast(b,c,\alpha):= \left\{x \in (0,1): a_n(x) \geq c^{b^{n^\alpha}} i.o.\right\},
\]
where $i.o.$ denotes infinitely often. The following result is a special case of Theorem 4.2 in Wang and Wu \cite{lesWW08}.

\begin{lemma}(\cite[Theorem 4.2]{lesWW08})\label{WW}
Let $b,c>1$ be a real number. Then
\[
\dim_{\rm H}E^\ast(b,c,\alpha)=
\begin{cases}
1/2,  &\text{$\alpha \in (0,1)$};\\
1/(b+1),  &\text{$\alpha=1$};\\
0,  & \text{$\alpha \in (1,+\infty)$}.
\end{cases}
\]
\end{lemma}

Now we show that $E(b,c,\alpha,\beta)$ is a subset of $F(d,c,\alpha)$ with $b>d>1$.

\begin{lemma}\label{subset}
Let $b,c >1$ and $\alpha,\beta>0$. Then for any $b>d>1$, we have $E(b,c,\alpha,\beta) \subseteq E^\ast(d,c,\alpha)$.
\end{lemma}

\begin{proof}
Since $b^{(n+1)^\alpha} - b^{n^\alpha} \to +\infty$ for any $\alpha>0$, we know that $c^{b^{(n+1)^\alpha} - b^{n^\alpha}} \geq c$ for sufficiently large $n$.
Choose $0<\delta <\beta$ enough small such that $(\beta-\delta)\cdot c > \beta +\delta$.

For any $x \in E(b,c,\alpha,\beta)$, we know that
\[
\lim_{n \to \infty} \frac{T_n(x)}{c^{b^{n^\alpha}}}= \beta.
\]
Note that
\begin{align*}
\frac{(\beta-\delta)c^{b^{(n+1)^\alpha}} - (\beta+\delta)c^{b^{n^\alpha}}}{c^{d^{(n+1)^\alpha}}} &= \frac{c^{b^{(n+1)^\alpha}}}{c^{d^{(n+1)^\alpha}}} \cdot \left(\beta-\delta - \left((\beta+\delta)\cdot \frac{b^{n^\alpha}}{b^{(n+1)^\alpha}}\right)\right) \\
& \geq c^{b^{(n+1)^\alpha} - d^{(n+1)^\alpha}} \cdot \left(\beta-\delta - (\beta+\delta)\cdot \frac{1}{c}\right) \to \infty,
\end{align*}
then there exists $N_0>0$ (depending on $\delta$) such that for any $n \geq N_0$, we have
\begin{equation*}
(\beta-\delta)c^{b^{n^\alpha}} \leq T_n(x) \leq (\beta+\delta)c^{b^{n^\alpha}}\ \text{and}\ \
(\beta-\delta)c^{b^{(n+1)^\alpha}} - (\beta+\delta)c^{b^{n^\alpha}} \geq c^{d^{(n+1)^\alpha}}.
\end{equation*}
Thus, we actually deduce that
\[
a_{n+1}(x) \geq T_{n+1}(x) - T_n(x) \geq (\beta-\delta)c^{b^{(n+1)^\alpha}} - (\beta+\delta)c^{b^{n^\alpha}} \geq c^{d^{(n+1)^\alpha}}
\]
holds for any $n \geq N_0$. So we have $x \in F(d,c,\alpha)$.
\end{proof}

It follows from Lemma \ref{WW} that $\dim_{\rm H}E(b,c,\alpha,\beta) \leq 1/2$ for $0<\alpha<1$ and $\dim_{\rm H}E(b,c,\alpha,\beta) =0$ for $\alpha>1$.
When $\alpha =1$, by Lemmas \ref{WW} and \ref{subset}, we know that
\[
\dim_{\rm H}E(b,c,\alpha,\beta) \leq \dim_{\rm H}F(d,c,\alpha) \leq 1/(d+1)
\]
for any $b >d>1$. So $\dim_{\rm H}E(b,c,\alpha,\beta) \leq 1/(b+1)$ by letting $d \to b^+$. Therefore, we eventually obtain that
\[
\dim_{\rm H}E(b,c,\alpha,\beta) \leq
\begin{cases}
1/2,  &\text{$\alpha \in (0,1)$};\\
1/(b+1),  &\text{$\alpha=1$};\\
0,  & \text{$\alpha \in (1,+\infty)$}.
\end{cases}
\]

\subsection{Lower bound}

For any $n \in \mathbb{N}$, we define
\begin{equation*}
f(n) = \left(\beta - \frac{1}{n}\right) c^{b^{n^\alpha}}\ \ \ \ \text{and} \ \ \ \ \
g(n) = \left(\beta + \frac{1}{n}\right) c^{b^{n^\alpha}}.
\end{equation*}
Then $f(n) \to \infty$ and $g(n) \to \infty$ as $n \to \infty$.
So we can choose $N >0$ sufficiently large such that
\[
f(n) \geq 2, \ g(n)\ \text{is non-decreasing and}\ \ \frac{ c^{b^{n^\alpha}}}{n} \geq 2
\]
for all $n \geq N$. Let
\[
F_{N}(b,c,\alpha,\beta) = \big\{x \in (0,1): f(n) \leq a_n(x) \leq g(n)\ \text{for all}\ n \geq N\big\}.
\]

The following lemma states that $F_{N}(b,c,\alpha,\beta)$ is a subset of $E(b,c,\alpha,\beta)$.

\begin{lemma}\label{subset2}
$F_{N}(b,c,\alpha,\beta) \subseteq E(b,c,\alpha,\beta)$.
\end{lemma}

\begin{proof}
For any $x \in F_{N}(b,c,\alpha,\beta)$, we have
\begin{equation}\label{fg}
f(n) \leq a_n(x) \leq g(n)
\end{equation}
for all $n \geq N$. Note that $g(n)$ is non-decreasing for all $n \geq N$, we know that
\[
T_n(x) \leq \max\big\{T_{N-1}(x), g(n)\big\} \leq T_{N-1}(x) +g(n).
\]
Combing this with (\ref{fg}), we deduce that
\[
f(n) \leq a_n(x)\leq T_n(x) \leq T_{N-1}(x) +g(n)
\]
for any $n \geq N$. Therefore,
\[
\lim_{n \to \infty} \frac{T_n(x)}{c^{b^{n^\alpha}}}= \beta
\]
and hence $x \in E(b,c,\alpha,\beta)$. That is to say, $F_{N}(b,c,\alpha,\beta) \subseteq E(b,c,\alpha,\beta)$.
\end{proof}

Next we estimate the lower bound for the Hausdorff dimenson of $F_{N}(b,c,\alpha,\beta)$. To do this, we need the following lemma, which provides a method to obtain a lower bound Hausdorff dimension of a fractal set (see \cite[Example 4.6]{lesFal90}).


\begin{lemma}\label{4.6}
Let $E = \bigcap_{n \geq 0} E_n$, where $[0,1] = E_0 \supset E_1 \supset \cdots$ is a decreasing sequence of subsets in $[0,1]$ and $E_n$ is a union of a finite number of disjoint closed intervals (called $n$-th level intervals) such that each interval in $E_{n-1}$ contains at least $m_n$ intervals of $E_n$ which are separated by gaps of lengths at least $\varepsilon_n$. If $m_n \geq 2$ and $\varepsilon_{n-1}> \varepsilon_n >0$, then
\[
\dim_\mathrm{H} E \geq \liminf_{n \to \infty} \frac{\log(m_1m_2 \cdots m_{n-1})}{-\log (m_{n}\varepsilon_n)}.
\]
\end{lemma}

\begin{proof}
Suppose the liminf is positive, otherwise the result is obvious. We may assume that each interval in $E_{n-1}$ contains exactly $m_n$ intervals of $E_n$ since we can remove some excess intervals to get smaller sets $E^\prime_n$ and $E^\prime$. Thus the gaps of lengths between different intervals in $E_n$ are not changed and hence we just need to deal with these new smaller sets. Now we define a mass distribution $\mu$ on $E$ by assigning a mass of $(m_1\cdots m_n)^{-1}$ to each of $(m_1\cdots m_n)$  $n$-th level intervals in $E_n$. Next we will check the conditions of the classical Mass distribution principle.

Let $U$ be an interval of length $|U|$ satisfying $0<|U|<\varepsilon_1$. Then there exists a integer $n$ such that $\varepsilon_n\leq |U|<\varepsilon_{n-1}$.
On the one hand, $U$ can intersect at most one $(n-1)$-th level interval since the gap of length between different intervals in $E_{n-1}$ is at least $\varepsilon_{n-1}$ and hence $U$ can intersect at most $m_n$ $n$-th level intervals; on the other hand, since $|U| \geq \varepsilon_n$ and the gap of length between different intervals in $E_{n}$ is at least $\varepsilon_{n}$, we know $U$ can intersect at most $(|U|/\varepsilon_n+1) \leq 2|U|/\varepsilon_n$ $n$-th level intervals. Note that each $n$-th level interval has mass $(m_1\cdots m_n)^{-1}$, so
\begin{align*}
\mu(U) &\leq \min\{m_n, 2|U|/\varepsilon_n\} \cdot (m_1\cdots m_n)^{-1}
\leq m^{1-s}_n\cdot (2|U|/\varepsilon_n)^s \cdot (m_1\cdots m_n)^{-1}
\end{align*}
for any $0 \leq s \leq 1$ and hence
\[
\frac{\mu(U)}{|U|^s} \leq \frac{2^s}{(m_1\cdots m_{n-1})m^s_n\varepsilon^s_n}.
\]
Let $s >\liminf\limits_{n \to \infty} \log(m_1 \cdots m_{n-1})/(-\log (m_{n}\varepsilon_{n}))$. Then $(m_1\cdots m_{n-1})m^s_n\varepsilon^s_n >1$ for sufficiently large $n$. Thus $\mu(U) \leq C\cdot |U|^s$, where $C>0$ is an absolute constant.
By the classical Mass distribution principle (see \cite[Chapter 4]{lesFal90}), we have $\dim_{\rm H} E \geq s$. This completes the proof.
\end{proof}

\begin{lemma}\label{FN}
\[
\dim_{\rm H}F_{N}(b,c,\alpha,\beta) \geq
\begin{cases}
1/2,  &\text{$\alpha \in (0,1)$};\\
1/(b+1),  &\text{$\alpha=1$};\\
0,  & \text{$\alpha \in (1,+\infty)$}.
\end{cases}
\]
\end{lemma}

\begin{proof}
For any $n \geq 1$, we define
\begin{align*}
\mathcal{D}_n = \Big\{(\sigma_1,\cdots,\sigma_n) \in\mathbb{N}^n:~1 \leq \sigma_k \leq &c^{b^{k^\alpha}}/k +1\ \text{for all}\ 1 \leq k <N\\
&  \text{and}\ f(k) \leq \sigma_k \leq g(k)\ \text{for all}\ N \leq k \leq n\Big\}
\end{align*}
and
\[
\mathcal{D} = \bigcup_{n \geq 0} \mathcal{D}_n
\]
with the convention $\mathcal{D}_0 := \emptyset$. For any $n \geq 1$ and $(\sigma_1,\cdots,\sigma_n) \in \mathcal{D}_n$, we denote
\[
J(\sigma_1,\cdots,\sigma_n) = \bigcup_{\sigma_{n+1}} cl I(\sigma_1,\cdots,\sigma_n, \sigma_{n+1})
\]
and call it the $n$-th level interval, where the union is taken over all $\sigma_{n+1}$ such that $(\sigma_1,\cdots,\sigma_n, \sigma_{n+1}) \in \mathcal{D}_{n+1}$, cl denotes the closure of a set and $I(\sigma_1,\cdots,\sigma_n, \sigma_{n+1})$ is the $(n+1)$-th cylinder for continued fractions.

Let
\[
E_0 = [0,1],\ \ \ \ E_n = \bigcup_{(\sigma_1,\cdots,\sigma_n)\in \mathcal{D}_n} J(\sigma_1,\cdots,\sigma_n)
\]
for any $n \geq 1$ and $E:= \bigcap_{n \geq 0} E_n$. Then $E$ is a Cantor-like subset of $F_{N}(b,c,\alpha,\beta)$. It follows from the construction of $E$ that each element in $E_{n-1}$ contains some number of the $n$-th level intervals in $E_n$. We denote such a number by $M_n$. If $1 \leq n < N$, by the definition of $\mathcal{D}_n$, we have
\[
M_n = \left\lfloor\frac{c^{b^{n^\alpha}}}{n}\right\rfloor +1,
\]
where $\lfloor x\rfloor$ denotes the greatest integer not exceeding $x$.
When $n \geq N$, we get
\[
M_n \geq \left\lfloor\frac{2c^{b^{n^\alpha}}}{n} \right\rfloor \geq \left\lfloor\frac{c^{b^{n^\alpha}}}{n}\right\rfloor +1.
\]
So we obtain $M_n \geq m_n :=\left\lfloor c^{b^{n^\alpha}}/n\right\rfloor +1$.

Next we estimate the gaps between the same order level intervals. For any $n \geq N$ and two distinct level intervals $J(\tau_1,\cdots,\tau_n)$ and $J(\sigma_1,\cdots,\sigma_n)$ of $E_n$, we assume that $J(\tau_1,\cdots,\tau_n)$ locates in the left of $J(\sigma_1,\cdots,\sigma_n)$ without loss of generality. By Proposition \ref{CY}, we know the level intervals $J(\tau_1,\cdots,\tau_n)$ and $J(\sigma_1,\cdots,\sigma_n)$ are separated by the $(n+1)$-th cylinder $I(\tau_1,\cdots,\tau_n,1)$ or $I(\sigma_1,\cdots,\sigma_n,1)$ according to $n$ is even or odd. In fact, if $n$ is odd,
note that $J(\tau_1,\cdots,\tau_n)$ is a union of a finite number of the closure of $(n+1)$-th order cylinders like $I(\tau_1,\cdots,\tau_n, j)$ with $2 \leq f(n+1) \leq j \leq g(n+1)$ and these cylinders run from right to left, and so is $J(\sigma_1,\cdots,\sigma_n)$, therefore $J(\tau_1,\cdots,\tau_n)$ and $J(\sigma_1,\cdots,\sigma_n)$ are separated by $I(\tau_1,\cdots,\tau_n,1)$ in this case. When $n$ is even, they are separated by $I(\sigma_1,\cdots,\sigma_n,1)$.
Thus the gap is at least
\[
|I(\tau_1,\cdots,\tau_n,1)| \ \ \text{or}\ \ |I(\sigma_1,\cdots,\sigma_n,1)|,
\]
where $|\cdot|$ denotes the length of a interval. In view of (\ref{DX}) and (\ref{cylinder inequality}), we deduce that
\begin{align*}
|I(\tau_1,\cdots,\tau_n,1)| \geq \frac{1}{2q^2_{n+1}} &\geq \frac{1}{8} \cdot \left(\prod_{k=1}^n (\tau_k +1)\right)^{-2}\\
& = \frac{1}{2^{2n+3}} \cdot \left(\prod_{k=1}^{N-1} \left( \frac{c^{b^{k^\alpha}}}{k} +1\right) \prod_{k=N}^n \left(\beta +\frac{1}{k}\right) c^{b^{k^\alpha}}\right)^{-2}\\
&:= \varepsilon_n.
\end{align*}
Similarly, we can also obtain $|I(\sigma_1,\cdots,\sigma_n,1)| \geq \varepsilon_n$.
It is easy to check that $\varepsilon_n > \varepsilon_{n+1} >0$ for sufficiently large $n$ and $\varepsilon_n \to 0$ as $n \to \infty$.
These imply that the gaps between any two $n$-th level intervals are at least $\varepsilon_n$.
By Lemma \ref{4.6}, we have
\begin{align*}
\dim_\mathrm{H} E &\geq \liminf_{n \to \infty} \frac{\log(m_1m_2 \cdots m_n)}{-\log (m_{n+1}\varepsilon_{n+1})} \\
& \geq \liminf_{n \to \infty} \frac{b^{1^\alpha}+b^{2^\alpha}+\cdots+b^{n^\alpha}}{ 2(b^{1^\alpha}+\cdots+b^{(n+1)^\alpha}) -  b^{(n+1)^\alpha}}\\
& = \liminf_{n \to \infty} \frac{b^{1^\alpha}+\cdots+b^{n^\alpha}}{ 2(b^{1^\alpha}+\cdots+b^{n^\alpha}) +  b^{(n+1)^\alpha}}.
\end{align*}
When $0 < \alpha <1$, then $b^{(n+1)^\alpha}/(b^{1^\alpha}+\cdots+b^{n^\alpha}) \to 0$ as $n \to \infty$ and hence $\dim_\mathrm{H} E \geq 1/2$. If $\alpha =1$, we know $b^{n+1}/(b+\cdots+b^n)\to b-1$ as $n \to \infty$ and hence $\dim_\mathrm{H} E  \geq 1/(b+1)$. When $\alpha >1$, we have $b^{(n+1)^\alpha}/(b^{1^\alpha}+\cdots+b^{n^\alpha}) \to +\infty$ as $n \to \infty$ and hence $\dim_\mathrm{H} E  \geq 0$.
Since $E$ is a subset of $F_{N}(b,c,\alpha,\beta)$, we complete the proof.

\end{proof}

Combing Lemmas \ref{subset2} and \ref{FN}, we get
\[
\dim_{\rm H} E(b,c,\alpha,\beta) \geq  \dim_{\rm H} F_{N}(b,c,\alpha,\beta) \geq
\begin{cases}
1/2,  &\text{$\alpha \in (0,1)$};\\
1/(b+1),  &\text{$\alpha=1$};\\
0,  & \text{$\alpha \in (1,+\infty)$}.
\end{cases}
\]
The results on the upper and lower bounds imply the proof of Theorem \ref{doubly exponential}.

\begin{remark}
In fact, from the proof of the lower bound, we can obtain that
\[
\dim_{\rm H} E_\phi \geq \liminf_{n \to \infty} \frac{\log\phi(1)+\cdots+\log\phi(n)}{ 2(\log\phi(1)+\cdots+\log\phi(n)) +  \log\phi(n+1)}.
\]
Therefore, we always have $\dim_{\rm H} E_\phi \geq 1/2$ when $\phi(n)$ tends to infinity with single exponential rates since 
\[
\limsup_{n \to \infty}\frac{(n+1)^\alpha}{1^\alpha + 2^\alpha +\cdots+n^\alpha} = 0
\] 
for any $\alpha >0$. And we also get $\dim_{\rm H} E(b,c,\alpha,\beta) \geq 1/2$ when $0 < \alpha < 1$ and $\dim_{\rm H} E(b,c,\alpha,\beta) \geq 1/(b+1)$ if $\alpha =1$ since 
\[
\limsup_{n \to \infty}\frac{b^{(n+1)^\alpha}}{b^{1^\alpha} + b^{2^\alpha} +\cdots+ b^{n^\alpha}} = 0
\] 
for $0 < \alpha < 1$ and it is $(b-1)$ when $\alpha =1$. Moreover, it also indicates that the Hausdorff dimension of $E_\phi$ is just related to the second base (i.e., $b$) in the doubly exponential rate when $\alpha =1$.
\end{remark}

\end{document}